\documentclass[11pt]{amsart}
\usepackage[usenames,dvipsnames]{pstricks}
\usepackage[mathscr]{eucal}
\usepackage{amsfonts,amsmath,amssymb,amsthm,amscd
	,amsxtra}
\usepackage{enumerate,verbatim}
\usepackage[all,2cell,ps]{xy}
\usepackage[notcite, notref]{}
\usepackage[pagebackref]{hyperref}
\usepackage{todonotes}
\usepackage{calc,color}
\usepackage{wasysym}
\usepackage{mathptmx}
\theoremstyle{plain} 

\newtheorem{thm}{Theorem}[section]
\newtheorem{cor}[thm]{Corollary}
\newtheorem{lem}[thm]{Lemma}
\newtheorem{nota}[thm]{Notation}

\newtheorem{prop}[thm]{Proposition}
\newtheorem{defn}[thm]{Definition}

\newtheorem{exam}[thm]{Example}
\newtheorem{rem}[thm]{Remark}
\newtheorem{ques}[thm]{Question}

\newcommand{\Ann}{\mbox{Ann}\,}

\newcommand{\Hom}{\mbox{Hom}\,}
\newcommand{\Ext}{\mbox{Ext}\,}
\newcommand{\e}{\mbox{e}\,}

\newcommand{\Tor}{\mbox{Tor}\,}

\newcommand{\Spec}{\mbox{Spec}\,}

\newcommand{\Ass}{\mbox{Ass}\,}

\newcommand{\Supp}{\mbox{Supp}\,}

\newcommand{\depth}{\mbox{depth}\,}
\renewcommand{\dim}{\mbox{dim}\,}

\newcommand{\q}{\mbox{q}\,}

\newcommand{\Tr}{\mbox{Tr}\,}

\newcommand{\redu}{\mbox{r}\,}

\newcommand{\re}{\mbox{red}\,}

\newcommand{\dep}{\mbox{depth}\,}

\newcommand{\lo}{\longrightarrow}
\newcommand{\su}{\subseteq}

\newcommand{\fm}{\mathfrak{m}}

\newcommand{\fq}{\mathfrak{q}}

\newcommand{\C}{C}

\begin{document}
\bibliographystyle{amsplain}


\title[Ulrich modules : decomposition and relation with rings having minimal multiplicity]
 {Ulrich modules : decomposition and relation with rings having minimal multiplicity}

\bibliographystyle{amsplain}

     \author[M. Rahimi]{Mehran Rahimi}

	\email{mehran\_rahimy@yahoo.com}
	
	\email{mehran.rahimy65@gmail.com}


\maketitle

\begin{abstract}
 Some basic properties of Ulrich modules are studied and it is proved that  Ulrich modules in Cohen-Macaulay local rings are decomposable. As a result, some characterizations proved for rings having minimal multiplicity, and some properties for ideals which are Ulrich as module is discussed.
\end{abstract} 

\section{Introduction}

\emph{Ulrich modules} are first introduced in \cite{BHU} as \emph{maximally generated maximal Cohen-Macaulay modules}. When the base ring is local, maximally generated maximal Cohen-Macaulay modules defined as maximal Cohen-Macaulay modules with the same multiplicity, with respect to maximal ideal, and minimal number of generators. Also these modules are known as the name \emph{linear maximal Cohen-Macaulay modules}.  These modules  are less known. For example, it is not known that if every Cohen-Macaulay local ring has a maximal Cohen-Macaulay Ulrich module or not. There is a good convention on some basic properties of Ulrich modules in \cite[Chapter 2]{AGR}. Also Authors try to investigate properties of Ulrich modules in \cite{GOTWY1} and classify them over two-dimensional rational singularities in \cite{GOTWY2} and over minimal multiplicities in  \cite{KT}.

In chapter 2, one investigate some basic properties, examples and remarks about  Ulrich modules. We focus on a special case : maximal-primary ideals in one-dimensional Cohen-Macaulay local rings, as they are maximal Cohen-Macaulay as $R$-modules.

In chapter 3 author prove some fundamental lemmas about minimal generating set of Ulrich modules and finally gives a decomposition for Ulrich modules, which is the main purpose of this paper :

\begin{thm}\ref{main}Assume that $(R,\fm)$ is  a local ring,  $C$ is an Ulrich $R$-module and $P_1, \cdots , P_s$, where $s>1$, are associated primes of $C$. Then   $C=\oplus_{i=1}^s\Gamma_{P_i}(C)$.
\end{thm} 

This Theorem can reduce study of Ulrich modules to the case that base ring is a domain, or has just one associated prime.
As a result, one can see that if $R$ is a one-dimensional Cohen-Macaulay local ring then powers of maximal ideal are decomposable from somewhere if $|\Ass(R)|>1$. In particular if $R$ has minimal multiplicity such that $\Ass(R)$ is not singleton then $R$ has decomposable maximal ideal. This result, and other facts  comes in chapter three. 

chapter four discuss about some results cause by our main Theorem. First, some small facts about \emph{generalized Ulrich modules}, and then some notations about ideals in  one-dimensional Cohen-Macaulay local rings which are Ulrich as $R$-modules. We introduce a new ideal $\Gamma_I := \oplus_{P\in\Ass(R)}\Gamma_{P}(R)$ and use it to characterize rings with minimal multiplicity:  

\begin{thm}\ref{mainGamma}
		Assume that $(R,\fm)$ is a one-dimensional Cohen-Macaulay local ring, with infinite residue field, such that $\Ass(R)$ is not singleton. The following conditions are equivalent:
	\begin{itemize}
		\item[(1)] $R$ has minimal multiplicity.

		\item[(2)] There exists prime ideals $P,T\in\Ass(R)$ which are Ulrich as $R$-modules.
		
		\item[(3)] $\Gamma_P(R)$ is Ulrich as an $R$-module for all $P\in\Ass(R)$. 
	\end{itemize}
	When this is the case, $R$ has decomposable maximal ideal
\end{thm}  

Next one is characterization with first syzygy of canonical module. Remember that, $R$ is almost Gorenstein if possessed a canonical module $K$ and exact sequence $0\lo R \lo K \lo c \lo 0$ such that $C$ is Ulrich, see \cite{AGR}. Note that $\dim(C) = \dim(R)-1$ this case. Next we investigate some similar exact sequence for rings having minimal multiplicity :  
\begin{thm}\ref{Uprop}
		Assume that $(R,\fm)$ is a $d$-dimensional $(d>0)$ Cohen-Macaulay local ring with infinite residue field, and that $R$ admits a canonical module $K$. The following conditions are equivalent.
	\begin{itemize}
		\item [(i)] $R$ has minimal multiplicity.
		\item [(ii)] $\Omega^1(K)$ is Ulrich.
		\item [(iii)] There exist natural number $t$ and exact sequence $0 \lo R \lo K^{t} \lo D \lo 0$ such that $D$ is maximal Cohen-Macaulay Ulrich module.
	\end{itemize}
when this is the case $t = \redu(R)$.
\end{thm}
 Author end chapter by a proposition about rings with canonical reductions, whom admits ideals which are Ulrich as $R$-modules, and some open questions rises from discussions of this paper.

In this paper by local we mean  a Noetherian local ring and all modules are finitely generated. Let $(R,\fm)$ be a local ring. For an $R$-module $M$, the symbols $\e_{\fm}^0(R)$ and $\Gamma_I(M)$  denotes respectively the multiplicity, with respect to $\fm$, and the local cohomology module of $M$ with respect to ideal $I$.

\section{Preliminaries}

In the rest of this section $(R,\fm)$ is a local ring such that $\dim(R) = d$.

\begin{defn}\emph{Let $M$ be an finitely generated $R$-module. Then $M$ is said to be Ulrich if $M$ is Cohen-Macaulay and $\mu(M) = \e^0_{\fm}(M)$.  }
\end{defn}

If $\dim(M) = 0$ then $M$ is Ulrich if and only if $\fm M = 0$. If $\dim(M) = s > 0$ and residue field of $R$ is infinite, then $M$ is Ulrich if and only if $\fq M = \fm M$ for some parameter $\fq$ of $M$. One can show that, image of such parameter in $R/\Ann(M)$ generates a minimal reduction for maximal ideal (see \cite[Proposition 2.2.]{AGR}). In the rest of this chapter we assume that the residue field of $R$ is infinite.

First author shows a simple equivalence in definition of Ulrich modules, which may be useful.

\begin{lem}\label{FirstLemma}
	Assume that $(R,\fm)$ is a Cohen-Macaulay local ring.  Consider the followings for a Cohen-Macaulay $R$-module $M$ :
	\begin{itemize}
		\item[(i)] $\mu(M)=\e_{\fm}^0(M)$.
		\item[(ii)] $\e_{\fm}^0(M)=\redu(M)$.
		\item[(iii)] $\redu(M)=\mu(M)$. 
	\end{itemize}
	Then $(i)\Leftrightarrow(ii)\Rightarrow(iii)$. 
\end{lem}
\begin{proof}
	Choosing a suitable parameter ideal $\fq$ for $M$, which is also a reduction for maximal ideal of $R/\Ann(M)$, one can easily see that $\e_{\fm}^0(M) = M/\fq M$ and $\redu(M) = \Hom_R(R/\fm , M/\fq M) = (\fq M:_M \fm)/\fq M$.
	
	$(i)\Rightarrow(ii)$. By assumption $\fq M = \fm M$ which means $\redu(M)= (\fq M: \fm)/\fq M = (\fm M:_M \fm)/\fm M = M/\fm M = \mu(M) = \e_{\fm}^0(M)$.
	
	$(ii)\Rightarrow(i)$. By assumption $(\fq M: \fm)/\fq M = M/\fq M$ which means $M = (\fq M:_M \fm)$ and so $\fm M = \fq M$.
	
	$(ii)\Rightarrow(iii)$. By assumption $M/\fq M = (\fq M:_M \fm)/\fq M$ which means $M = (\fq M:_M \fm)$ and so $\fm M=\fq M$. Therefore $\redu(M) = M/\fq M = M/\fm M = \mu(M)$.
	
\end{proof}

So, if $M$ is Ulrich, one have $\mu(M)=\e_{\fm}^0(M) =\redu(M)$. If $R$ is a singular Gorenstein ring then $\mu(R) = \redu(R)=1$ where $\e_{\fm}^0(R)>1$. This example shows that all three conditions in above proposition are not equivalent.

\begin{rem}\label{1Rem} \emph{Let $R$ be a one-dimensional Cohen-Macaulay local ring. Then all non-zero ideals are maximal Cohen-Macaulay modules and one may concern when they are Ulrich as an $R$-module and what is the properties. Some of these propertes may use freely in coming chapters. }
\begin{itemize}	
		\item[(1)]  $\fm^n$ is Ulrich as an $R$-module for $n\geq \re(\fm)$.
		\item[(2)] If $I$ is Ulrich then so is trace of $I$, and $JI$ for any  ideal $J$ of $R$.
	\item[(3)] $I$ is $\fm$-primary and Ulrich if and only if $\mu(I) = \e_{\fm}^0(R)$.
	
	\item[(4)] If $I$ is Ulrich then $I$ is $\fm$-full. More precise, $(\fm I : x)=(\fm I : \fm) = I$ for some superficial element $x\in\fm$. Also, if $I\su xR$ then   $I=\fm(I:\fm)$.
	\item[(5)] If $I$ is Ulrich and also reduction of $J$ then $J^n$ is Ulrich for $n>>0$.
	
	\item[(6)] If $I$ is Ulrich and contains a superficial element for $R$ then $I=\fm$.
	
	\item[(7)] Let $K$ be a canonical ideal of $R$ and $I$ an arbitrary ideal such that $K \subseteq I$. Then, $(K:I)$ is Ulrich if and only if $I$ is Ulrich. 
	
	\item[(8)] if $I$ is $\fm$-primary and $\redu(R/I) = \e_{\fm}^0(R)$ then $I$ is Ulrich.
\end{itemize}

\end{rem} 
\begin{proof} Let $x\in \fm$ be a superficial element. Hence, if $I$ is Ulrich as an $R$-module then $\fm I = xI$.
	
$(1)$. If $r=\re(\fm)$ then	$\fm(\fm^n) = x(\fm^n)$ for $n\geq r$, which shows the result.
	
$(2)$. If $y\in I$ be a non-zero-divisor then $I(y:I)=y\Tr_R(I)$. Now	$y\fm\Tr_R(I) = \fm I(y:I) = xI(y:I) = yx\Tr_R(I)$ which follows the result. Also multiplying $J$ to $\fm I=xI$ gives us the Ulrichness of $JI$.

$(3)$. If $I$ is Ulrich as an $R$-module then $\fm I = xI$ and $\mu(I) = \l(I/\fm I) = \l(I/xI) = \e_{\fm}^0(R)$.  Conversely we have $I/\fm I = \mu(I) = \e_{\fm}^0(R) = I/xI$  which means $\fm I = xI$.

$(4)$. First one can see that $(\fm I : x) = (xI : xR) = I$. Then $I \subseteq (\fm I : \fm) = (\fm I : xR) = I$ which shows other qualities.  Also, if $I = xJ$  for some ideal $J$ of $R$ then $\fm(I:\fm) =\fm(xJ:\fm) = \fm(\fm J : \fm) = \fm J = xJ = I$.

$(5)$. Clear as $J^n= IJ^{n+1}$ for $n>>0$. Now multiplying with $\fm$ and $xR$ gives the result.

$(6)$. Let $xI = \fm I$ while $x\in\fm$. Therefore we got $xI \subseteq x\fm \su I\fm = xI$ which means $I=\fm$. 

$(7)$. By \cite[Lemma 2.1]{CanRed} one has $\mu(K:I) = \redu(I)$. Now by $(2)$, $(K:I)$ is Ulrich if and only if $\mu(K:I)= \e_{\fm}^0(R)$, if and only if $\redu(I) = \e_{\fm}^0(R)$, if and only if $I$ is Ulrich by Lemma\ref{FirstLemma}.

$(8)$. Applying $\Hom_R(R/\fm , -)$ over $0\lo I \lo R \lo R/I \lo 0$ we embedding $\Hom_R(R/\fm , R/I) \lo \Ext_R^1(R/\fm , I)$. Therefore, as $\Ext_R^1(R/\fm , I)\cong  \Hom_R(R/\fm , I/xI) \cong (xI:_I \fm)/xI \su I/xI$, one get $\e_{\fm}^0(R) =\redu(R/I) =\l(\Hom_R(R/\fm , R/I))\leq \l(I/xI) = \e_{\fm}^0(R)$. Now the result follows by Lemma\ref{FirstLemma}(ii). 
\end{proof}

 \begin{exam}
 \emph{	Let $(R,\fm)$ be a one-dimensional Cohen-Macaulay local ring. Remark\ref{1Rem} shows that powers of maximal ideal are Ulrich as an $R$-module from somewhere. Multiplying by other ideals and taking dual with canonical module, if exists, gives new Ulrich modules. One famous example is the case $R$ is    analytically unramified ring. Then the integral closure $\bar{R}$ of $R$ is finite as an $R$-module. Set $\C = (R : \bar{R})$, the conductor ideal. Then both $\bar{R}$ and $\C$ are Ulrich modules.    }
 \end{exam}

Next we show some basic facts about Ulrich modules and ideals which are Ulrich as $R$-modules.  

\begin{lem}\label{Lemx}
	Assume that $(R,\fm)$ is a   Cohen-Macaulay local ring with infinite residue field and that $M$ is a one-dimensional  Cohen-Macaulay$R$-module. If $T$ is an Ulrich submodule of $M$, Then  for any non-zero-divisor $x\in\fm$ there exists a sub-module $N$ of $M$ such that $T\subseteq N$, $N\nsubseteq xM$, $N$ is Ulrich  $R$-module and $\mu(N)=\mu(T)$.  
\end{lem}
\begin{proof}
If $T\nsubseteq xR$ then there is nothing to proof. Assume that $T\subseteq xM$ and  Let $T=xT_1$ for some submodule $T_1\in M$. By definition of Ulrich modules, it immediately conclude that $T_1$ is Ulrich. If $T_1 \nsubseteq xM$ we are done. Otherwise there exists  $T_2 \subseteq M$ such that $T_1 = xT_2$ and we have chain of submodules $T \subseteq T_1 \subseteq T_2 \subseteq \cdots $. As $R$ is Noetherian, this chain must stop somewhere which means $xT_{s+1} = T_s =T_{s+1}$ and so $T_{s+1}=0$ and consequently $T=0$. Therefore, to avoid contradiction, there must be natural number $r$ such that $T_r \nsubseteq xM$.
\end{proof}
\begin{prop}\label{PropMuType}
Assume that $M$ is a one-dimensional Cohen-Macaulay $R$-module. Let $N$ be a submodule of $M$ such that $M/N$ is again one-dimensional Cohen-Macaulay $R$-module. 
\begin{itemize}
\item [(i)] If $M$ is Ulrich then so is $N$ and $M/N$.
\item [(ii)] If $N$ is Ulrich then $\mu(N)\leq\redu(M)$.
\end{itemize}
. In particular if $I$ is an ideal of a one-dimensional Cohen-Macaulay local ring $R$, which is Ulrich as an $R$-module, such that $R/I$ is one-dimensional Cohen-Macaulay, then $ \mu(I)\leq \redu(R)$.
\end{prop}
\begin{proof} $(i)$. Choose $x\in\fm$ regular over $R$ and $M$ and $M/N$. Applying $\Tor^R(R/xR , -)$ over exact sequence $0 \lo N \lo M \lo M/N \lo 0$ we get  $0 \lo N/xN \lo M/xM \lo M/xM+N \lo 0$ which follows the result.

$(ii)$. Let $x\in R$ be a non-zero-divisor for both $M$, $M/N$ and $R$. Proof of $(i)$. shows that $N\nsubseteq xM$. Now,  By definition                   $\redu(M) = \l(\Ext_R^1(R/\fm , M)) = \l(\Hom_R(R/\fm , M/xM)) = \l((xM :_M \fm)/xM)$.  Therefore
\begin{center}
	$\redu(M)=\l((xM :_M \fm)/xM) \geq \l((N+xM)/xM) =\l(N/xM\cap N)= \mu(N)$,
\end{center}	
where the last equality comes from the facts that $xM\cap N = xN$ and  $N$ is Ulrich.
\end{proof}

First, we investigate about $\fm$-primary and Cohen-Macaulay ideals which are Ulrich as $R$-module.

\begin{prop}\label{Lemtype}
Assume that $(R,\fm)$ is a one-dimensional Cohen-Macaulay local ring with infinite residue field. Assume that $I$ is an  ideal of $R$ which is Ulrich as an $R$-module. Then:
\begin{itemize}
	\item[(i)] If $I$ is $\fm$-primary then $\e_{\fm}^0(R)\leq\redu(R)+\redu(R/I)$. In particular, if $R/I$ is Gorenstein then $R$ has minimal multiplicity. 
	\item[(ii)] If $R/I$ is again one-dimensional Cohen-Macaulay ring then  $ \mu(I)\leq \redu(R) $. 
\end{itemize} 
In particular, if $R$ is Gorenstein then  equality hold in both cases. Therefore, in $(ii)$ the ideal $I$ is cyclic.
\end{prop}
\begin{proof}
Note that choosing $x\in\fm$ as a superficial element for $R$  we have $I\fm=xI$ by assumption. Also, by Lemma\ref{Lemx} we may assume that $I\nsubseteq xR$. Therefore one have
\begin{center}
	$\redu(R)=\l((x:\fm)/xR) \geq \l((I+xR)/xR) =\l(I/xR\cap I)= \mu(I) - \l((I:\fm)/I) $.
\end{center}  
where the last equality comes from the exact sequence $0\lo x(I:\fm)/xI \lo I/xI \lo I/x(I:\fm)\lo 0$ and the fact that $(I:\fm)=(I:xR)$ since $I$ is Ulrich as an $R$-module.

$(i)$. The inequality is clear as $\mu(I)=\e_{\fm}^0(R)$ by Lemma\ref{FirstLemma}.

$(ii)$. The inequality is clear as $\l((I:\fm)/I) =0$ when $R/I$ is Cohen-Macaulay.

In particular, if $R$ is Gorenstein then $\fm = (x:I)$ which means $x:\fm=x:(x:I)=I + xR$ and equality holds.
\end{proof}

Next, we see what happens when one of the associated primes is Ulrich as an $R$-module.

\begin{prop}\label{Ptype}
	Assume that $(R,\fm)$ is a one-dimensional Cohen-Macaulay local ring with infinite residue field. Assume that $P\in\Spec(R)$ such that $P$ is Ulrich as an $R$module. Then
	\begin{itemize}
		\item [(i)] $\e_{\fm}^0(R) \leq \redu(R)+\e_{\fm}^0(R/P)$.
		\item[(ii)]  if  $R/P$ is regular then $R$ has minimal multiplicity.
		\item[(iii)] if $R$ is Gorenstein then $R$ is hypersurface and $\Ass(R) = \{ P \}$.
	\end{itemize}

\end{prop}
\begin{proof}
	$(i)$. Consider the exact sequence $0\lo P\lo R\lo R/P\lo0$ which give us $\e_{\fm}^0(R)=\e_{\fm}^0(P) + \e_{\fm}^0(R/P)$. Now the result follows by Proposition\ref{PropMuType}.
	
	$(ii)$. Clear as then $\redu(R) + 1 \leq \e_{\fm}^0(R) \leq \redu(R)+\e_{\fm}^0(R/P) \leq  \redu(R) + 1$
	
	$(iii)$. By Proposition\ref{Lemtype} we  have $\e_{\fm}^0(P) = \mu(P) = \redu(R) = 1$. Now $\e_{\fm}^0(P) =\l(PR_P)\e_{\fm}^0(R/P) + \Sigma_{q\neq p}\l(R_q)\e_{\fm}^0(R/q)$ where Sigma runs over other associated primes of $R$. Therefore $R/P$ is regular and $P$ is the only associated prime of $R$ and $\e_{\fm}^0(R) = 2$ by $(i)$, which shows the result.
\end{proof}

\begin{rem} \emph{Assume that $(R,\fm)$ is a one-dimensional Cohen-Macaulay local ring with infinite residue field, and Assume that $P\in\Spec(R)$. One may see Proposition\ref{Ptype} in another way: If $R/P$ is regular then there exists $x\in\fm$ such that $\fm= xR + P$. Now, if  $P\subseteq xR$  then $\fm^2=(P+xR)^2 \subseteq xR$ which means $R$ has minimal multiplicity. In particular, if $\Ass(R) = \{ P \}$ such that $R/P$ is regular and $P^2=0$ then $R$ has minimal multiplicity.}
	
\end{rem}

Let $(R,\fm)$ be a one-dimensional Cohen-Macaulay local ring. It is clear that when $\fm$ is Ulrich then $R$ has minimal multiplicity, means $\e_{\fm}^0(R)=\redu(R)+1$. Next proposition prepare some upper bound for multiplicity of the ring $R$ when some $\fm$-primary ideal, not necessary maximal ideal, is Ulrich as an $R$-module.

\begin{prop}\label{PropType}Let $R$ be a  one-dimensional Cohen-Macaulay local ring and assume that $\Ass(R)$ is not singleton. Let $I$ be an $\fm$-primary ideal of $R$ which is Ulrich as an $R$-module. If	 $I=I_1\oplus I_2$ such that $R/I_1$ and $R/I_2$ are one-dimensional Cohen-Macaulay. Then $1 + \redu(R) \leq\e_{\fm}^0(R)\leq 2\redu(R)$.

\end{prop}
\begin{proof}
By Proposition\ref{Lemtype} we have $\mu(I_1) \leq\redu(R)$ and $\mu(I_2)\leq \redu(R)$. Therefore
\begin{center}
	$\e_{\fm}^0(R) = \mu(I) = \mu(I_1)+\mu(I_2)\leq 2\redu(R)$.
\end{center}
\end{proof}

\section{Decomposability of Ulrich modules}

\begin{prop}\label{2A}Let $(R,\fm)$ be a local Noetherian ring with infinite residue field. Let $C$ be a Cohen-Macaulay $R$-module of dimension $d>0$ and that $P \in \Ass(C)$. If $\q$ is a regular sequence for $C$ then one has $\q C\cap\Gamma_{P}(C) = \q\Gamma_{P}(C)$.
\end{prop}
\begin{proof}
	Proof is by induction on $d$. First assume that $d=1$ and $x\in\fm$ is regular for $C$. Let $a\in xC\cap\Gamma_{p}(C)$. Therefore, there exists $b\in C$ such that $a=xb$. As $a\in\Gamma_{P}(C)$, there exists $t\in N$ such that $P^t xb=0$ which means $b\in \Gamma_{P}(C)$ since $x$ is regular over $C$. 
	
	Now assume that $d>1$ and inductive hypothesis holds for $s<d$. Assume that $\q=x_1R+x_2R + \cdots +x_sR$ is a regular sequence for $C$. Let $a = \Sigma_{i=1}^sx_ic_i \in \q C\cap\Gamma_{P}(C)$ and set $\q_j=\Sigma_{i=1,i\neq J}^sx_iR$. As $C/x_iC$ is $d-1$-dimensional Cohen-Macaulay, we got $ \q_i(C/x_iC)\cap\Gamma_{P}(C/x_iC)=\q_i\Gamma_{P}(C/x_iC)$ for all $i$ and so, by choosing a suffice large $t$,  we have $P^tc_i \in x_jC$ for all $j$ such that $j\neq i$. Therefor $P^tc_i \in \cap_{j=1,j\neq i}^{s}x_jC = \prod_{j=1,j\neq i}^sx_jC$. Let $P^tc_i= \prod_{j=1,j\neq i}^sx_jN_i$ for $N_i\subseteq C$ for all $i$. Therefore, we got
	\begin{center}
	$0=P^ta=\Sigma_{i=1}^sP^tx_ic_i=\Sigma_{i=1}^sx_i\prod_{j=1,j\neq i}^sx_jN_i =\Sigma_{i=1}^s\prod_{j=1}^sx_jN_i = \prod_{j=1}^sx_j(\Sigma_{i=1}^sN_i)$.			
	\end{center}
As $\prod_{j=1}^sx_j$ is regular element of $C$, must have $\Sigma_{i=1}^sN_i = 0$ which means $N_i=0$ for all $i$; and so $P^tc_i=0$ for all $i$. Now the proof is complete.
\end{proof}

\begin{cor}\label{genGamma}
	
	Assume that $C$ is a $d$-dimensional Ulrich module and $P\in \Ass(C)$. Then any minimal generating set for $\Gamma_{P}(C)$ is a part of a minimal generating set for $C$.
\end{cor}
\begin{proof}
	
	It is enough to prove that $\fm C\cap\Gamma_{P}(C)=\fm\Gamma_{P}(C)$ for  $P\in \Ass(C)$. Assume that $\q$ is a regular sequence for $C$ such that $\fm C=\q C$. Therefore 
	\begin{center}
		
	$	\fm C\cap\Gamma_{p}(C) =\q C\cap\Gamma_{P}(C) =\q\Gamma_{P}(C) \subseteq \fm\Gamma_{P}(C)$.
	\end{center}
\end{proof}

\begin{nota}\label{n1}\emph{
It is a well-known fact that, for an $R$-module $C$ and $P\in \Ass(C)$, $\Ass(\Gamma_{P}(C))=P$ and $\Ass(C/\Gamma_{P}(C))=\Ass(C) -\{ P \}$. Therefore $C_P = \Gamma_{PR_P}(C_P)$. Now, let $C$ be a Cohen-Macaulay $R$-module of dimension $d$ and $\Ass(C)= \{ P_1, P_2, \cdots, P_t\}$. It is straightforward to see that $\Gamma_{P_i}(C) \cap\Gamma_{P_j}(C)=0$ for all $i$ and $j$ such that $i\neq j $ and $1\leq i,j \leq t$. Therefore the summation $\Gamma_{P_1}(C)+\Gamma_{P_2}(C)+\cdots \Gamma_{P_t}(C)$ is direct. Set $G=\oplus_{i=1}^t\Gamma_{P_i}(C)$ as submodule of $C$ and consider the exact sequence $0\lo G\lo C \lo C/G \lo 0$. Then $P_i \notin \Ass(C/G)$ for all $i$, $1\leq i\leq t$ which means $\dim(C/G) < \dim(C)$. There is two different cases:
\begin{itemize}
	\item[(i)] $d=1$. Then $G$ is also Cohen-Macaulay of dimension one and $\Ass(C/G) \subseteq \{ \fm \}$.	
	\item[(ii)] $d>1$. This case, as $\dep(C/\Gamma_{P_i}(C))>0$ for all $P_i\in\Ass(C)$, we have $\dep(C/G)>0$. This means if $\dim(C)=2$ then all $\Gamma_{P_i}(C)$s are Cohen-Macaulay.
\end{itemize}}
\end{nota}
\begin{lem}\label{tor}
	Assume that $R$ is a local ring and  $M$ is an $R$-module such that $\dep(R)=\dim(M)>0$. If $\q$ is a regular sequence of length $s$ for $R$ such that $\Tor_1^R(R/\q,M)=0$ then $\dep(M)\geq s$. 
\end{lem}
\begin{proof}
	Let $\q=x_1R+x_2+\cdots x_sR$ be a regular sequence for $R$. Set $\q'=x_1R+\cdots+x_{s-1}R$. Applying $M\otimes_R-$ over exact sequence $0\lo R/\q' \overset{\eta}{\lo} R/\q' \lo R/\q \lo 0$, with $\eta$ multiplicative map by $x_s$, we get long exact sequence
	\begin{center}
$\Tor_1^R(R/\q',M)\overset{\phi}{\lo}\Tor^R_1(R/\q',M)\lo\Tor_1^R(R/\q,M)\lo M/\q'M \overset{M\otimes \eta}{\lo}M/\q'M\lo M/\q M\lo 0$.		
	\end{center}
By assumption $\Tor_1^R(R/\q,M)=0$ which means 
\begin{itemize}
	\item[(i)] $M\otimes\eta$ is one to one which means so $x_s$ is regular over $M/\q'M$ and so $\dep(M)\geq1$. 
	\item[(ii)] $\phi$ is onto which means $\Tor^R_1(R/\q',M)=0$ as it is multiplicative by $x_s$.
\end{itemize}
Continuing this way, starting from $\q'$ instead of $\q$ in above discussion, the result will follows. 
\end{proof}
\begin{lem}\label{Uldep}
	Assume that $C$ is an Ulrich $R$-module and $P\in\Ass(C)$. Then, $\Gamma_{P}(C)$ is Ulrich if and only if it is Cohen-Macaulay. In particular, if $\dep(R) \geq \dim(C)$ then $\Gamma_{P}(C)$ is Ulrich.
\end{lem}
\begin{proof}
	The only if part is clear as every Ulrich module is Cohen-Macaulay. For if part, let $\Gamma_{P}(C)$ be Cohen-Macaulay and $\q$ a superficial sequence for $C$, which means $\fm C=\fq C$. Since $\q$ is also a regular $\Gamma_{P}(C)$, it is enough to show that $\fm\Gamma_{P}(C)\subseteq\q\Gamma_{P}(C)$. Applying $R/\q$ over exact sequence $0\lo \Gamma_{P}(C)\lo C\lo C/\Gamma_{P}(C)\lo 0$ we get long exact sequence
	\begin{center}
		$0\lo \Tor_1^R(R/\q,C/\Gamma_{P}(C)) \lo \Gamma_{P}(C)/\q\Gamma_{P}(C) \overset{\eta}{\lo}  C/\q C\lo C/\q C+\Gamma_{P}(C)\lo 0$.
	\end{center} 
	Note that as $\Gamma_{P}(C)\cap\q C=\q\Gamma_{P}(C)$ by---, $\eta$ is one to one and so $\Tor_1^R(R/\q,C/\Gamma_{P}(C))=0$. Also, as $\fq C=\fm C$, if $a\in\fm\Gamma_{P}(C)$ then $\eta(a)\in\fm C=\q C$ which means $a\in\q\Gamma_{P}(C)$. 
	
	In particular, assume that $\dep(R)\geq\dim(C)$. Then $\dep(R)\geq\dim(C/\Gamma_{P}(C))$ and, as $\Tor_1^R(R/\q,C/\Gamma_{P}(C))=0$, by Lemma\ref{tor} we have $C/\Gamma_{P}(C)$ is Cohen-Macaulay; either $\Gamma_{P}(C)$. 
\end{proof}
Now we are ready to main Theorem of this section.
\begin{thm}\label{main}Assume that $C$ is an Ulrich $R$-module of dimension $d>0$ and $\Ass(C) = \{ P_1 , P_2 , \cdots , P_s \}$. Then $\oplus_{i=1}^s\Gamma_{P_i}(C)$ minimally generated by a part of a minimal generating set for $C$. If $\dep(R)\geq\dim(C)$ then  $C=\oplus_{i=1}^s\Gamma_{P_i}(C)$.
	
\end{thm}
\begin{proof}
	By Corollary\ref{genGamma} and Notation\ref{n1}, $\oplus_{1}^s(\Gamma_{P_i}(C))$ generates minimally by part of a minimal generating set for $C$. Set $G=\oplus_{1}^s(\Gamma_{P_i}(C))$. If $C= \oplus_{1}^s(\Gamma_{P_i}(C))$ then we are done. Otherwise, choose $a_1\in C$ such that $a_1\notin\fm C\cup\oplus_{1}^s(\Gamma_{P_i}(C))$. Continuing this way, we get $a_1, \cdots, a_t$ such that $C=\oplus_{1}^s(\Gamma_{P_i}(C))+\Sigma_{j=1}^ta_j$. So, we may set $N=a_1R+\cdots +a_tR$.
	
	 If $\dep(R)\geq\dim(C)$ then $\Gamma_{P_i}(C)$ is Ulrich for all $P_i\in\Ass(C)$. If $\q$ is a regular sequence of maximal length for $C$ which is also a regular sequence for $R$, applying $R/\q\otimes-$ over $0\lo G\lo C\lo C/G\lo0$ we get $\Tor_1^R(R/\fq,C/G)=0$ by the same discussion as in proof of Lemma\ref{Uldep}. Therefore, by Lemma\ref{tor} we have $\dep(C/G) \geq \dim(C)$ which is a contradiction since $\dim(C/G)<\dim(C)$. Therefore $C/G=0$ and the result follows.

\end{proof}
\begin{cor}
	Assume that $C$ is Ulrich $R$-module of dimension $d>0$ and that $\Ass(D) = \{ P_1 , P_2 , \cdots , P_s \}$. Let $N$ be submodule of $C$ such that  $C = (\oplus_{i=1}^s\Gamma_{P_i}(C))+N$ minimally by method describes in proof of Theorem\ref{main}. Then
	\begin{itemize}
		\item[(i)] $\mu(C) = \Sigma_{i=1}^s\mu(\Gamma_{P_i}(C)) +\mu(N) $. In case $N\neq0$ we have $\mu(C/N)=\Sigma_{i=1}^s\mu(\Gamma_{P_i}(C))$. In particular $\mu(C) \geq s$.
	 \item[(ii)] If $R$ is Cohen-Macaulay and $C$ an Ulrich $R$-module then $C=\oplus_{i=1}^s\Gamma_{P_i}(C)$.
	\end{itemize}
\end{cor}
\begin{proof}
	(i). Note that if $N\neq0$ then $\mu(C/N) = \oplus_{i=1}^s\mu((\Gamma_{P_i}(C)/(\oplus_{i=1}^s\Gamma_{P_i}(C))\cap N)$. By the method of choosing $N$, $\oplus_{i=1}^s\Gamma_{P_i}(C)\cap N\subseteq \fm \oplus_{i=1}^s\Gamma_{P_i}(C)$ and so is non-generator.
	
	(ii). If $R$ is Cohen-Macaulay then $\dep(R)=\dim(R)\geq\dim(C)$. Now the result follows by Theorem\ref{main}.
\end{proof}

\begin{cor}
	Assume that $(R,\fm)$ is a one-dimensional Cohen-Macaulay local ring, with infinite residue field, and that $\Ass(R)$ is not singleton. If $R$ has minimal multiplicity then  $R$ has decomposable maximal ideal.
\end{cor}
\begin{proof}
	As $R/\fm$ is infinite, there exists superficial element $x\in\fm$ such that $\fm^2 =x\fm$ which means $\fm$ is Ulrich as an $R$-module. Now the result follows by Theorem\ref{main}
\end{proof}

As it is clear, the corollary above does not work when $R$ is a domain; like numerical semigroup rings.

\begin{prop}\label{UlAll}
	Assume that $C$ is an Ulrich $R$-module of dimension $d$ and $\Ass(C)=\{P_1, P_2, \cdots, P_s\}$. Then $C=\oplus_{i=1}^s(\Gamma_{P_i}(C))$ if and only if $\Gamma_{P_i}(C)$ is Ulrich for all $i$.
\end{prop}
\begin{proof}
First let $C=\oplus_{i=1}^s(\Gamma_{P_i}(C))$. Then for all $i$, $0\leq i\leq d-1$, we have 
\begin{center}
	$ 0=\Ext^i_R(R/\fm,C) = \Ext^i_R(R/\fm,\oplus_{i=1}^s(\Gamma_{P_i}(C))) \cong \oplus_{i=1}^s\Ext^i_R(R/\fm,\Gamma_{P_i}(C))$
\end{center}
which means $\Gamma_{P_i}(C)$ is Cohen-Macaulay and so Ulrich by Lemma\ref{Uldep}, for all $i$, $1\leq i\leq s$.

Now, for "if" part, let $\Gamma_{P_i}(C)$ is Ulrich for all $i$. Therefore $\e^0_{\fm}(\Gamma_{P}(C))=\mu(\Gamma_{P}(C))$ for all $P\in\Ass(C)$. By Theorem\ref{main}, there exists submodule $N$ of $C$ such that $C=\oplus_{i=1}^s(\Gamma_{P_i}(C))+N$ and so $\mu(C)=\Sigma_{i=1}^s(\mu(\Gamma_{P_i}(C)))+\mu(N)$. Now, by multiplicity formula we have
\begin{center}
	$\mu(C)=\e_{\fm}^0(C) = \Sigma_{i=1}^s\ell(C_P)\e_{\fm/P}^0(R/P) = \Sigma_{i=1}^s\e_{\fm}^0(\Gamma_{P_i}(C))=\Sigma_{i=1}^s\mu(\Gamma_{P_i}(C)) $
\end{center}
which means $N=0$ and the result follows.
\end{proof}

\section{Applications}

\begin{lem}\label{C}Assume that $(R,\fm)$ is a local ring and that $C$ is a finitely generated $R$-module such that $C$ has no embed associated prime. Let $P \in \Ass(C)$. The following conditions are equivalent.
	\begin{itemize}
		\item[(a)]$(PC)R_P = 0$, it means $C_P$ is Ulrich.  
		\item [(b)] $P\Gamma_P(C) = 0$, it means $\Gamma_P(C) = (0:_C P)$. 
	\end{itemize}
\end{lem}
\begin{proof}
	(b)$\Rightarrow$(a) is clear as $C_P = \Gamma_{PR_P}(C_P)$.
	
	(a)$\Rightarrow$(b). Assume contrarily that there exists $c \in C$ such that $P^ic = 0$ but $Pc \neq 0$. Therefore, there exist $t$, $1<t<i$, such that $P^{t+1}c = 0$ but $P^tc \neq 0$, which means there exist $r \in P^t$ such that $rc \neq 0$ and we have $Prc = 0$. Hence $P \subset \Ann(rc)$ and, since $C$  has no embedded associated primes, we have $P = \Ann(rc)$. Hence
	$P \in \Ass(P^iC) \subset \Ass(PC)$ which means $(PC)_P \neq 0$, a contradiction.

\end{proof}

\begin{prop}
	Assume that $(R,\fm)$ is a local ring and that $C$ is a finitely generated $R$-module such that $C$ has no embed associated prime. Let $\Ass(C)=\{ P_1,\cdots P_s\}$. Then followings are equivalent.
	\begin{itemize}
		\item [(i)] $C_P$ is Ulrich for all $P\in\Ass(C)$.
		\item [(ii)] $P_1P_2 \cdots P_s \subset \Ann(C)$.
	\end{itemize}
	In particular  localization of $R$-modules with radical annihilator ideal in their associated primes is Ulrich.
\end{prop}
\begin{proof}
	$(i)\Rightarrow(ii)$ is easy to gain by localizing $P_1P_2 \cdots P_s\Ann(D)=0$ at associated primes of $D$.
	
	$(ii)\Rightarrow(i)$, Assume that $(PD)R_P = 0$ for each $P \in \Ass(D)$. Since $\Ass(D)$ has no embedded prime, we have ${P_i}R_{P_j} = R_{P_j}$ when $i \neq j$. Hence
	for each $j$, $1\leq j \leq s$, $(P_1P_2\cdots P_sD)_{P_j} = {P_1}_{P_j}{P_2}_{P_j}\cdots {P_s}_{P_j}D_{P_j} = {P_j}_{P_j}D_{P_j} = 0$ by assumption and $P_j$ does not belong to $\Supp(P_1P_2\cdots P_sD)$. But $\Ass(P_1P_2\cdots P_sD) \subset \Ass(D)$, so $\Ass(P_1P_2\cdots P_sD)$ is empty and we have $(P_1P_2\cdots P_sD)
	= 0$.
\end{proof}

\begin{cor}\label{C1}
	Assume that $(R,\fm)$ is a local ring and $C$ is an Ulrich $R$-module. The followings are equivalent:
	\begin{itemize}
		\item[(i)] $\Ann(C) $ is a radical ideal.
		
		\item[(ii)] for each $P \in \Ass(C)$ we have $(PC)R_P = 0$, i.e. $C_P$ is Ulrich.
		
	\end{itemize} 
	In particular, faithful Ulrich modules are generically Ulrich if and only if the base ring is reduced.
\end{cor}
\begin{proof} Let $\Ass(C)=\{P_1,\cdots, P_s\}$.
	
	$(i)\Rightarrow(ii)$. The result follows by Lemma\ref{C} since  $P_1P_2 \cdots P_sC=0$. Now, localizing at associated
	
	$(ii)\Rightarrow(i)$. By Theorem\ref{main} we have $C=\Gamma_{P_i}^s(C)+N$. Therefore, by Lemma\ref{C} we have \begin{center}
		$(0 : C) = (0 : \Gamma_{P_i}^s(C)+N)=\cap_{i}^s(0 : \Gamma_{P_i}(C))\cap(0 : N)=\cap_{i}^sP_i\cap N$	
	\end{center}
	which follows the result as $(0 : N)\subseteq\cap_{i}^sP_i=\sqrt{\Ann(C)}$.
\end{proof}

\begin{nota}
	Assume that $R$ is a one-dimensional Cohen Macaulay local ring, with infinite residue field, which have minimal multiplicity and that $\Ass(R)$ is not singleton. Set $|\Ass(R)|=s$. 
	\begin{itemize}
		\item[(i)] If $\redu(R)=1$ then $s=2$ and $R$ is fiber product of two regular ring. If $\redu(R)=2$ then $R$ is fiber product of rings which at least one of them is regular.
		\item[(ii)] If $\redu(R)\geq s$ then $R$ is fiber product of local rings which at least associated prime of one them is singleton.
		\item[(iii)] $R$ is fiber product of regular rings if and only if $\redu(R)+1=s$.
	
	\end{itemize}
To proof and other conventions, one may consult with \cite[Fact 2.6]{NT}
\end{nota}
Next Remark, shows some elementary properties of ideals which are Ulrich as an $R$-module and make some statements which will be needed next. 
\begin{rem}
	\emph{Assume that $(R,\fm)$ is a one-dimensional Cohen-Macaulay local ring. Therefore,  every non-zero ideal of $R$ is a one-dimensional Cohen-Macaulay module. Let $\Ass(R) = \{ P_1, P_2, \cdots P_s\}$ for $s>1$, and $I$ be an ideal of $R$. For $R$-module $M$, set $\Gamma_M := \oplus_{P\in\Ass(M)}\Gamma_{P}(M)$. Clearly $\Gamma_R := \oplus_{i=1}^s\Gamma_{P_i}(R)$ is an $\fm$-primary ideal of $R$ and contains $\Gamma_I$. Also $\Gamma_R = \Gamma_{\fm}$ unless $|\Ass(R)|$ is singleton. Here comes some basic facts about ideals of $R$ which are Ulrich as an $R$-module.}
	\begin{itemize}
		\item[(1)] If $I$ is Ulrich  Then $I\subseteq\Gamma_R$.  
		\item[(2)]  $I$ is $\fm$-primary if and only if $\Gamma_I$ is $\fm$-primary.
	
		\item[(3)] If $P\in\Ass(R)$ then $\Gamma_P=\Gamma_P(P)\oplus_{q\neq P}\Gamma_{q}(R)$.

	\end{itemize}
\end{rem}
\begin{proof}
	$(1)$. By Theorem\ref{main}, we have $I=\oplus_{i=1}^s\Gamma_{P_i}(I) \subseteq \oplus_{i=1}^s\Gamma_{P_i}(R)$. Note that $\Gamma_P(I)=I\cap\Gamma_P(R)$.
	
	$(2)$. If part is clear as $\Gamma_I \subseteq I$. For only if part, let $I$ be $\fm$-primary and $\Gamma_I \subseteq P$ for some $P\in\Ass(R)$. Then, localizing $\Gamma_P(I) \subseteq P$ at $P$ we get $R_P = I_P = \Gamma_{PR_P}(I_P) \subseteq PR_P$, which is a contradiction.
	
	$(3)$. Note that if $q\in\Ass(R)$ and $\alpha\in\Gamma_{q}(R)$ then there exists natural number $t$ such that $q^t\alpha=0$. Now, if $q \neq P$ then $\alpha\in P$ since $q^t\nsubseteq P$. Therefore $\Gamma_{q}(P)=\Gamma_{q}(R)$ if $q\neq P$.
\end{proof}
 

\begin{prop}\label{PUlrich}	Assume that $(R,\fm)$ is a one-dimensional Cohen-Macaulay local ring, with infinite residue field, and that $\Ass(R)=\{P_1,P_2,\cdots P_s\}$ with $s>1$.
	\begin{itemize}
		\item[(i)] If $P\in\Ass(R)$ then $R/\Gamma_{P}(R)$ and $R/\Gamma_{P}(P)$ are Cohen-Macaulay rings.
		\item[(ii)] If $P\su \Gamma_R$ then $P=\Gamma_P$.
		\item[(iii)] If $I$ is Ulrich and $P\su I$ then $P$ is Ulrich.
	\end{itemize}
\end{prop}
\begin{proof}
	$(i)$. Let $a\in\fm$ such that $a\fm \su \Gamma_{P}(R)$. As $\Gamma_{P}(R)$ is finite, there exists natural number $n$ such that $P^n\Gamma_{P}(R)=0$. Therefore $P^na\fm = 0$, and so $P^na=0$ and $a\in\Gamma_{P}(R)$, which means $\fm\notin\Ass(R/\Gamma_{P}(R))$. Now if $a\fm \su \Gamma_{P}(P) \su P$ then by above discussion $a\in \Gamma_{P}(R)\cap P = \Gamma_{P}(P)$, and the result follows. 
	
	$(ii)$. Consider the exact sequence $0\lo P/\Gamma_P\lo \Gamma_R/\Gamma_P \lo \Gamma_R/P\lo0$. Note that $\Ass(P/\Gamma_P)\subseteq \{\fm\}$ and $\Gamma_R/\Gamma_P\cong \Gamma_P(R)/\Gamma_P(P)$ which is one-dimensional Cohen-Macaulay $R$-module by proof of $(i)$. Therefore $P=\Gamma_P$.
	
	$(iii)$. By Theorem\ref{main} one has $I=\oplus_{i=1}^s\Gamma_{P}(I)$.
	Therefore  inclusions
	$\Gamma_P=\Gamma_{P}(P)\oplus_{q\neq P}\Gamma_{q}(R) \subseteq P\subseteq \oplus_{i=1}^s\Gamma_{P}(I)$ yield that  $\Gamma_{q}(R)=\Gamma_{q}(I)$, and is Ulrich obviously, for all $q\in\Ass(R)$ with $q\neq P$; and we have exact sequence $0\lo \Gamma_P(P) \lo \Gamma_P(I)\lo \Gamma_P(I)/\Gamma_P(P)\lo0$. If $\alpha\fm \subseteq \Gamma_P(P)$ then $\alpha\in \Gamma_P(P)$ as $P$ is a prime ideal. Therefore $\depth(\Gamma_P(I)/\Gamma_P(P)) > 0$ and $\Gamma_P(P)$ is Ulrich by Proposition\ref{PropMuType}. Now consider the exact sequence $0\lo P/\Gamma_P\lo I/\Gamma_P \lo I/P\lo0$. Note that $\Ass(P/\Gamma_P)\subseteq \{\fm\}$ and $I/\Gamma_P\cong \Gamma_P(I)/\Gamma_P(P)$ which is one-dimensional Cohen-Macaulay $R$-module by same discussion as in proof of $(i)$. Therefore $P=\Gamma_P$ and is Ulrich
\end{proof}

\begin{rem}
	\emph{	Assume that $(R,\fm)$ is a one-dimensional Cohen-Macaulay local ring, with infinite residue field, such that $\Ass(R)=\{P_1,P_2,\cdots P_s\}$ with $s>1$. Following proof of Proposition\ref{PUlrich}, one can see if $P\su \Gamma_R$ then $P=\Gamma_P$, no need for $\Gamma_R$ to be Ulrich. It is not known that, if $P$ is Ulrich in this case or not.}
	
\end{rem}

\begin{thm}\label{mainGamma}
	Assume that $(R,\fm)$ is a one-dimensional Cohen-Macaulay local ring, with infinite residue field, such that $\Ass(R)$ is not singleton. The following conditions are equivalent:
	\begin{itemize}
		\item[(i)] $R$ has minimal multiplicity.
		\item[(ii)] $\fm$ is Ulrich as an $R$-module.
		\item[(iii)] There exists an $\fm$-primary ideal $I$, which is Ulrich as an $R$-module, such that $P , T\subseteq I$ for some $P , T\in\Ass(R)$. 
		\item[(iv)] There exists prime ideals $P,T\in\Ass(R)$ which are Ulrich as $R$-modules.
		\item[(v)] $\Gamma_P(R)$ is Ulrich as an $R$-module for all $P\in\Ass(R)$. 
	\end{itemize}
	When this is the case, $R$ has decomposable maximal ideal.

\end{thm}
\begin{proof}
	$(i)\Rightarrow(ii)$ is clear by definition of rings with minimal multiplicity.
	
	$(ii)\Rightarrow(iii)$. clear as one can take $I=\fm$.
	
	$(iii)\Rightarrow(iv)$. Follows immidiately by Proposition\ref{PUlrich}(iii). 
	
	$(iv)\Rightarrow(v)$.	As $P$ is Ulrich $\Gamma_{q}(R)$ is Ulrich for all $q\in\Ass(R)$ with $q\neq P$. Also, as $T$ is Ulrich $\Gamma_{P}(R)$ is Ulrich since $T\neq P$. Now the result follows.
	
	
	$(v)\Rightarrow(i)$. Consider $(\Gamma_R + xR)/xR \leq R/xR$. Note that if $x\alpha\in\Gamma_{P}(R)$ then $\fm x\alpha \subseteq \fm\Gamma_{P}(R)=x\Gamma_{P}(R)$ which means $\fm\alpha\subseteq\Gamma_{P}(R)$ and so $\alpha\in\Gamma_{P}(R)$ as $\fm$ is a regular ideal. This means $x\alpha$ is not a generator for $\Gamma_R$ and therefore $\mu((\Gamma_R + xR)/xR) = \mu(\Gamma_R)=\e_{\fm}^0(R))$. Now, as $\mu((\Gamma_R + xR)/xR)$ annihilated by $\fm$ we get $\e_{\fm}^0(R) - 1 \leq \mu((\Gamma_R + xR)/xR) \leq \redu(R/xR) =\redu(R) < \e_{\fm}^0(R) $ which complete the proof.
\end{proof}
This shows immediately that, every $d$-dimensional Cohen Macaulay ring $R$ with minimal multiplicity have quasi-decomposable maximal ideal provided $R$ has a superficial sequence $\q$ such that $\Ass(R/\q)$ is not singleton; see \cite{NT}. However, in case $d>1$, maximal ideals of rings with minimal multiplicity are not Cohen-Macaulay, and so outside the scope of Theorem\ref{main}.

Here comes another equivalence for rings having minimal multiplicity, in term of Ulrichness of first syzygy of canonical module. First we prove a lemma. 

\begin{lem}\label{Ulem}
	Assume that $(R,\fm)$ is a one-dimensional Cohen-Macaulay local ring, with infinite residue field, admitting canonical module $K$.
	\begin{itemize}
		\item [(a)] Assume that $T$ is an arbitrary maximal Cohen-Macaulay module. Then $T$ is Ulrich, if and only if $\Hom_R(T , K)$ is Ulrich.
		
		\item [(b)] Assume that there exist natural number $t$ and exact sequence $0 \lo R \overset{\phi}{\lo} K^t \lo D \lo 0$ such that $D$ is maximal Cohen-Macaulay Ulrich module. Then $ t = \redu(R)$, $\mu(D) = \redu(R)^2 - 1$ and $\Omega^1K$ is Ulrich.
	\end{itemize}
\end{lem}
\begin{proof}
	(a). Let $T$ be Ulrich. As $R/\fm$ is infinite, there exists superficial element $x \in \fm$ for $T$ such that $\fm T = xT$. Applying $\hom_R( - , K)$ over exact sequence $0 \lo T \overset{\times x}{\lo} T \lo T/xT \lo 0$ gives the exact sequence $ 0 \lo \Hom_R(T , K) \overset{\times x}{\lo} \Hom_R(T , K) \lo \Ext^1_R(T/xT , K) \lo 0$. Therefore $\fm$ annihilates $\Hom_R(T , K)/x\Hom_R(T , K)$ and so $\Hom_R(T , K)$ is Ulrich. For converse, just note that $\Hom_R(\Hom_R(T , K) , K) \cong K$. 
	
	(b). Set $r = \redu(R)$. Applying $\Hom_R( - , K)$ over $0 \lo R \overset{\phi}{\lo} K^t \lo D \lo 0$ gives the exact sequence $0 \lo \Hom_R(D , K) \lo R^t \lo K \lo 0$. Therefore $t \geq r$ and there is stable isomorphism $\Hom_R(D , K)\oplus R^r \cong \Omega^1K\oplus R^t$. As $R/\fm$ is infinite, we may choose $x \in \fm$ as a superficial element for $\Hom_R(D , K)$ which is also regular over $R$ and $\Omega^1K$. By tensor $R/xR$ in above isomorphism, we get the isomorphism 
	$$(\Hom_R(D , K)/x\Hom_R(D , K))\oplus (R^r/xR^r) \cong (\Omega^1K/x\Omega^1K)\oplus (R^t/xR^t).$$
	
	As $\Hom_R(D , K)$ is Ulrich, $\fm\Hom_R(D , K) = x\Hom_R(D , K)$. Hence, multiplying above sentences by $\fm$, gives the isomorphism
	
	$$\fm R^r/xR^r \cong (\fm\Omega^1K/x\Omega^1K)\oplus(\fm R^t/xR^t).$$   
	
	Therefore $t = r$, and $\Omega^1K \cong \Hom_R(D , K)$ is Ulrich. For final claim, let $\phi(1) \in K^r$. Applying $\Hom_R(D , K)$ over $0 \lo R \overset{\phi}{\lo} K^t \lo D \lo 0$ gives the exact sequence
	
	$$0 \lo \Hom_R(D , K) \lo \Hom_R(K^r , K) \overset{Hom_R(\phi,K)}{\lo} \Hom_R(R , K) \lo 0.$$
	
	Assume that $f \in \Hom_R(K^r , K)$ is a generator. Without loss of generality, we may assume that for all $a = (a_1, a_2, \ldots , a_r) \in K^r$ we have $f(a) = a_1$. Therefore $\Hom_R(\phi , K)(f)(1) = f(\phi(1)) = f\phi(1) \in \fm K$ which is a contradiction because $\Hom_R(\phi , K)$ is onto. Therefore $\phi(1) \notin \fm K^r$, and is a generator for $K^r$. Now, applying $R/\fm\otimes -$ over $0 \lo R \overset{\phi}{\lo} K^t \lo D \lo 0$ gives the exact sequence 
	
	$$ \Tor_1^R(R/\fm , D) \lo R/\fm \overset{\bar{\phi}}{\lo} K^r/\fm K^r \lo D/\fm D \lo 0.$$
	
	As $\phi(1) \notin \fm K^r$, $\bar{\phi}$ is not a zero function and so is one to one. Hence $\mu(D) = r\mu(K) - 1 = r^2 - 1$. 
\end{proof}

\begin{thm}\label{Uprop}
	Assume that $(R,\fm)$ is a $d$-dimensional $(d>0)$ Cohen-Macaulay local ring with infinite residue field, and that $R$ admits a canonical module $K$. The following conditions are equivalent.
	\begin{itemize}
		\item [(i)] $R$ has minimal multiplicity.
		\item [(ii)] $\Omega^1(K)$ is Ulrich.
		\item [(iii)] There exist natural number $t$ and exact sequence $0 \lo R \lo K^{t} \lo D \lo 0$ such that $D$ is maximal Cohen-Macaulay Ulrich module.
	\end{itemize} 
\end{thm}
\begin{proof} Without loss of generality, we may assume that $\dim(R) = 1$. 
	Set $r = \redu(R)$ and $e = \e_{\fm}^0(R)$. Applying $\Hom_R(- , K)$ over exact sequence $0 \lo \Omega^1K \lo R^r \lo K \lo 0$ 	gives the exact sequence $0 \lo R \overset{\phi}{\lo} K^r \lo \Hom_R(\Omega^1K , K) \lo 0$. 
	
	(i)$\Rightarrow$(iii). By Lemma \ref{Ulem} (b), $\phi(1)$ is a generator of $K^r$ and $\mu(\Hom_R(\Omega^1K , K)) = r^2 - 1$. On the other hand, $\e_{\fm}^0(\Hom_R(\Omega^1K ,  K)) = re - e = e(r - 1) = (r+1)(r-1) = r^2 - 1$ since $R$ has minimal multiplicity by assumption. Therefore $\Hom_R(\Omega^1K , K)$ is Ulrich and the result follows.

	(iii)$\Rightarrow$(i). By Lemma \ref{Ulem} (b), $t = r$ and $\mu(D) = r^2 - 1$. As $D$ is Ulrich, we have $r^2 - 1 = \mu(D) = \e_{\fm}^0(D) = e(r - 1)$ which means $e = r+1$ and so $R$ has minimal multiplicity. 
	
	(ii)$\Leftrightarrow$(iii).  By Lemma \ref{Ulem}(a), for a maximal Cohen-Macaulay module $T$, $T$ is Ulrich if and only if $\Hom_R(T , K)$ is Ulrich. Now, for (ii)$\Rightarrow$(iii)  apply $\Hom_R( - , K)$ over $0 \lo \Omega^1K \lo R^r \lo K \lo 0$. For (iii)$\Rightarrow$(ii), as by Lemma \ref{Ulem} (b) we have $t = r$, applying $\Hom_R( - , K)$ over $0 \lo R \lo K^t \lo D \lo 0$ gives the result.
\end{proof}

As we see, when the equivalent condition in Theorem \ref{Uprop} happens, $\mu(\Omega^1K) = r^2 - 1$. It is not known that, if this result holds, when $R$ do not have minimal multiplicity.

We end the paper with a Proposition about rings with canonical reductions having an $\fm$-primary ideal which is Ulrich as an $R$-module. Recall that a one-dimensional local ring  $R$ is called having canonical reduction if there exists canonical ideal for $R$ which is also a reduction for $\fm$, see \cite{CanRed} for more conventions. First we prove a lemma. Recall that if $R$ has a canonical ideal $K$, then $\rho(R) = \re(K)$ which is independent from choosing canonical ideal. $R$ is Gorenstein if and only if $\rho(R)=1$. Next interesting case is when $\rho(R) =2$. Remark\cite[Remark 3.4]{CanRed} would give a good background about properties of such rings.

\begin{prop}
	Assume that $(R,\fm)$ is a one-dimensional Cohen-Macaulay local ring admitting a canonical reduction $K$. Assume that there exists an ideal $I$ of $R$, such that $I$ is Ulrich as an $R$-module. 
	\begin{itemize}
			\item[(i)] If $y\in I\cap K$ be a non-zero-divisor then $(yK:I)=x(y:I)$. In particular $(yK:I)$ and $(y:I)$ are both Ulrich.
		\item[(ii)]  $\rho(R)=2$.
		\item[(iii)] $I$ is reflexive $R$-module.
	\end{itemize}
\end{prop}
\begin{proof}
		
	$(i)$. First note that $x(y:I)I=xy\Tr(I)\subseteq yK$, and so $x(y:I)\subseteq (yK:I)$. For the other inclusion, note that $(yK:I)=(xyK:xI)=(xyK:KI)$. Let $rKI\subseteq xyK$. As $\fm I=xI$ there exists $T\subseteq I$ such that $rI=xT$. Therefore $rKI=xTK\subseteq xyK$ which means $T \subseteq (yK:K) =yR$. Let $T=yT'$ for some ideal $T'$ of $R$. As $y\in I$ and $rI=xT=xyT'$ there exists $t'\in T'$ such that $ry=xyt'$ which means $r=xt'$. Therefore $(yK:I)\subseteq xR$. To complete the proof, we must show that $t'\in(y:I)$. As $r=xt'\in(yK:I)$ one has $t'KI=t'xI\subseteq yK$ and so $t' \in (yK:KI) =(y:I)$.
	
	$(ii)$.First assume that $I\subseteq K$ and choose $y\in I$ as a non-zero-divisor. By $(ii)$ we have $y\in (yK:I)=x(y : I)$ which means $y\in xR$. As $I$ can be generated by regular elements one may conclude that $I\subseteq xR$. Therefore, as by Lemma\ref{Lemx} there exists ideal which contains $I$ and not contained by $xR$ which is Ulrich as $R$-module,    
	we may assume hat $I \nsubseteq K$.  Note that $\fm=(K:I)$ as $I\nsubseteq K$. Hence $(K:\fm)=I+K$; and if $\fm^{n+1}=K\fm^n$ we got $KI=xI$. Now, if $y\in I$ be non-zero-divisor then there exists $I' \subseteq I$ such that $yK=xI'$ and so $I'$ is a canonical ideal of $R$ and $y\in I'$. By the method of choosing $I'$, we have $(x:K)=(y:I')$. As $I'K\subseteq IK\subseteq xR$ we get $I'\subseteq y:I'$ which means $\rho(R)=2$.
	
	$(iii)$. Let $y\in I$ be as in $(i)$. Then:
	\begin{center}
		$I=(yK:(yK:I) ) = (yK:x(y:I) ) = (yK : K(y:I) ) = (y:y:I)$.
	\end{center}
\end{proof}
Next is a characterization for rings with canonical reductions, whom admit $\fm$-primary ideal which are Ulrich as $R$-module, to be almost Gorenstein. Note that almost Gorenstein rings and nearly Gorenstein rings have canonical reduction, see\cite[Corollary 3.10]{CanRed}.
\begin{prop}
	Assume that $(R,\fm)$ is a one-dimensional Cohen-Macaulay local ring admitting a canonical reduction $K$. Assume that there exists an ideal $I$ of $R$, such that $I$ is Ulrich as an $R$-module. The followings are equivalent
	\begin{itemize}
		\item[(i)] If $K\subseteq I$, 
	\item[(ii)] $R$ is almost Gorenstein, 
	\item[(iii)] $I=\fm$. 
	
	\end{itemize} 
	When this is the case $R$, has minimal multiplicity and $\l(R/K) =2$.
	
\end{prop}
\begin{proof}Let $x\in K$ be a superficial element for $\fm$ and so $\fm I = xI$.
	
	$(i)\Rightarrow(ii)$. If $K\subseteq I$ then, as $\fm I=xI\subseteq K$, we have $\mu(\fm)=\mu(K:I)=\redu(I)=\e_{\fm}^0(R)$ and so $\fm$ is Ulrich and has minimal multiplicity. Therefore by \cite[Proposition 3.8(a)]{CanRed} $R$ is almost Gorenstein. 
	
$(ii)\Rightarrow(iii)$. Let $R$ be almost Gorenstein. Therefore $(x:K)=\fm$ and $K^2=x(x:\fm)$, see \cite[Remark 3.4]{CanRed}. Now, as $\fm=(x:I)=(x:K)$ we got $(xK:IK)=(xK:K^2)$ and so, as $xK$ is Gorenstein ideal, $IK=K^2$. Hence $IK=x(x:\fm)$. Multiplying by $\fm$, we get $x^2I=xIK=\fm IK=x\fm(x:\fm)=x^2\fm$ which means $I=\fm$.

$(iii)\Rightarrow(i)$ Is obvious.	
 
 	Finally, when this is the case, as $I=\fm$ we have $R$ has minimal multiplicity and so, again by \cite[Proposition 3.8(a)]{CanRed}, one has $\l(R/K) = 2$. 
\end{proof}

There are some questions, arising from discussions of this paper, for author, which have  no answer for now.

\begin{ques}
\emph{	Assume that $(R,\fm)$ is a one-dimensional Cohen-Macaulay local ring. If $\fm = \Gamma_R$ then, is $\fm$ Ulrich as an $R$-module?   }
\end{ques}

\begin{ques}
	\emph{	Assume that $(R,\fm)$ is a one-dimensional Cohen-Macaulay local ring. As we see in Theorem\ref{Uprop}, having minimal multiplicity can be formulate by Ulrichness of $\Omega^1(K)$ where $K$ is the canonical module of $R$. Therefore $\Tr_R(\Omega^1(K))$ is Ulrich this case. The question is, what can one say about $\Tr_R(\Omega^1(K))$? Is it equal to $\Gamma_R$? Is it equal to $\fm$ or $\Gamma_R$?   }
\end{ques}




\begin{thebibliography}{15}




\bibitem{BHU}
J. P. Brennan, J. Herzog and B. Ulrich, \emph{ Maximally generated Cohen–Macaulay modules}, Math. Scand. 61 (1987), 181–203.



\bibitem{GOTWY1}
S. Goto, K. Ozeki, R. Takahashi, K.-I. Watanabe, K.-I. Yoshida, \emph{Ulrich ideals and modules}, Math. Proc. Cambridge
Philos. Soc. 156 (2014), no. 1, 137–166.


\bibitem{GOTWY2}
S. Goto, K. Ozeki, R. Takahashi, K.-I. Watanabe and K.-I. Yoshida, \emph{Ulrich ideals and modules over two-dimensional rational singularities}, Nagoya Math. J. 221 (2016), no. 1, 69–110.

\bibitem{AGR}
S. Goto, R. Takahashi, and N. Taniguchi, \emph{Almost Gorenstein rings --towards a theory of higher dimension--}, J. Pure Appl. Algebra, 219 (2015), 2666-2712




\bibitem{KT}
T. Kobayashi and R. Takahashi, \emph{Ulrich modules ove Cohen-Macaulay local rings with minimal multiplicity}, The Quarterly Journal of Mathematics, Volume 70, Issue 2, June 2019, Pages 487–507

\bibitem{NT}
Saeed Nasseh, Sean Sather-Wagstaff, Ryo Takahashi and Keller VandeBogert,  \emph{Applications and homological properties of local rings with decomposable maximal ideals}, Journal of Pure and Applied Algebra, Volume 223, Issue 3, 2019, Pages 1272-1287, ISSN 0022-4049,

\bibitem{CanRed}
M. Rahimi, \emph{Rings with canonical reductions}, Bulletin of the Iranian Mathematical Society. 46. 1801-1825. 
\end{thebibliography}
\end{document}